\documentclass[a4paper,10pt]{amsart}
\usepackage[utf8]{inputenc}
\usepackage{amsmath, amssymb,bm}
\usepackage[all]{xy}
\usepackage{epigraph}

\title[Groups of type $\E_6$ and $\E_7$  and Related Torsors]{Groups of type $\E_6$ and $\E_7$ over Rings via Brown Algebras and Related Torsors}
\author[S.\ Alsaody]{Seidon Alsaody}
\address{Department of Mathematics, Uppsala University, P.\ O.\ Box 480, 751 06 Uppsala, Sweden. seidon.alsaody@math.uu.se}

\newtheorem{Thm}{Theorem}[section]
\newtheorem{Prp}[Thm]{Proposition}
\newtheorem{Cor}[Thm]{Corollary}
\newtheorem{Lma}[Thm]{Lemma} 
\theoremstyle{definition}
\newtheorem{Def}[Thm]{Definition}
\newtheorem{Rk}[Thm]{Remark}
\newtheorem{Ex}[Thm]{Example}
\numberwithin{equation}{section}

\newcommand{\C}{\mathbb{C}}
\newcommand{\Z}{\mathbb{Z}}

\newcommand{\Q}{\mathbb{Q}}
\newcommand{\D}{\mathrm{D}}
\newcommand{\E}{\mathrm{E}}
\newcommand{\F}{\mathrm{F}}
\newcommand{\G}{\mathrm{G}}

\newcommand{\diag}{\mathrm{diag}}

\newcommand{\Id}{\mathrm{Id}}

\newcommand{\bAut}{\mathbf{Aut}}

\newcommand{\bGL}{\mathbf{GL}}

\newcommand{\Spec}{\mathrm{Spec}}

\newcommand{\bX}{\mathbf{X}}

\newcommand{\bS}{\mathbf{S}}

\newcommand{\bE}{\mathbf{E}}

\newcommand{\bmu}{\bm{\mu}}

\newcommand{\bIsom}{\mathbf{Isom}}
\newcommand{\bInv}{\mathbf{Inv}}

\entrymodifiers={+!!<0pt,\fontdimen22\textfont2>}
\begin{document}

\begin{abstract} We study structurable algebras and their associated Freudenthal triple systems over commutative rings. The automorphism groups of these triple systems are exceptional groups of type $\E_7$, and we realize groups of type $\E_6$ as centralizers. When 6 is invertible, we further give a geometric description of  homogeneous spaces of type $\E_7/\E_6$, and show that they parametrize principal isotopes of Brown algebras. As opposed to the situation over fields or local rings, we show that such isotopes may be non-isomorphic.

\end{abstract}
\maketitle


\bigskip

\setlength\epigraphwidth{.72\textwidth}
\epigraph{``Les sym\'etries que l'on attribue \`a un objet d\'ependent de fa\c con essentielle des qualit\'es de l'objet que l'on d\'ecide de prendre en compte \`a l'exclusion de toute autre. Sans ce processus d'abstraction, aucune sym\'etrie (parfaite) n'est possible.''}{Jacques Tits, \textit{Sym\'etries} \cite{Tit}}

\section{Introduction}

Exceptional groups arise as symmetry groups of certain exceptional, nonassociative algebras. This is true for Lie groups, for algebraic groups over fields and more generally for affine group schemes over rings. Many of these algebras possess different kind of symmetries and, consequently, serve to represent different groups. Thus groups of type $\G_2$ are automorphism groups of octonion algebras, which are endowed with a quadratic form whose isometry group involves type $\D_4$. Groups of type $\F_4$ are automorphism groups of Albert algebras, whose cubic norms have isometry groups of type $\E_6$. 

This interplay between groups and algebras goes in both directions: on the one hand, the algebras afford concrete descriptions of certain torsors under exceptional groups; on the other hand, the yoga of torsors helps determine when two such algebras are isomorphic. This direction was used in \cite{Gil1}, where Gille proved that over rings, octonion algebras may have isometric norms without being isomorphic. In \cite{AG}, the author together with Gille used this interplay to give a concrete parameterization of all octonion algebras having the same norm. The key was the triality phenomenon, which endows groups of type $\D_4$ with the structure of $\G_2$ torsors in a particular way. The approach of \cite{AG} was continued in \cite{Als}, where the author studied a generalization of reduced Albert algebras and their coordinate algebras by realizing certain groups of type $\F_4$ as $\D_4$-torsors, showing that isomorphic algebras may admit non-isometric coordinate algebras. Further,  isotopes of Albert algebras were encoded as twists by an $\F_4$-torsor with total space a group scheme of type $\E_6$. The introduction of \cite{Als} gives a more detailed overview of the programme thus far.

In this paper, we take the next step of this programme of realizing exceptional groups as torsors over smaller exceptional groups. Namely, we study the interplay between groups of type $\E_6$ and $\E_7$. Our main tool will be \emph{Brown algebras}, which we generalize to rings. One way of obtaining Brown algebras starts with an Albert algebra. Brown algebras thus arising are said to be \emph{reduced}. Groups of type $\E_6$ enter as isometry groups of the cubic norms of Albert algebras, and as the identity components of automorphism groups of Brown algebras. 

Groups of type $\E_7$ arise as invariance groups of \emph{Freudenthal triple systems}. These are locally isomorphic (with respect to the fppf-topology) to the Freudenthal triple systems that can be constructed from reduced Brown algebras. In fact, such systems can be assigned to a wider class of Brown algebras. This brings about a number of natural questions:
\begin{enumerate}
 \item Are Brown algebras whose corresponding Freudenthal triple systems coincide, necessarily isomorphic as algebras?
 \item If not, can we parametrize Brown algebras having isomorphic Freudenthal triple systems?
 \item On the geometric side, can we give a concrete description of homogeneous spaces of type $\E_7/\E_6$?
\end{enumerate}

In what follows, we will be able to answer the first question in the negative, provide the parameterization desired in the second question under some invertibility assumptions, and, along the way, give the description desired in the third question.

Freudenthal triple systems (FTS) have been studied over fields of characteristic not 2 or 3; see e.g.\ \cite{Bro} and \cite{Gar}. The datum of an FTS consists, in this setting, of a vector space endowed with a bilinear form and either a trilinear map or a quartic form (either of the two determines the other) satisfying certain axioms. This definition of an FTS can be stated over rings, but problems may arise if 6 is not invertible. In the recent preprint \cite{GPR}, Garibaldi, Petersson and Racine gave a more elaborate definition, inspired by the work \cite{Lur} of Lurie. They showed that their definition gives objects with the right invariance group, and that when $6$ is invertible, one recovers the classical definition. We use the definition of \cite{GPR} and develop the basics of what we need without assuming that 6 is invertible. However, for the main results of the paper, we need to have access to the theory over fields. Since that theory assumes 6 to be invertible, so do we. Indeed, disposing of this assumption would require developing a large corpus of preparatory results over fields of characteristic 2 and 3, a task that is outside the scope of this paper. We hope that this paper leads to, or motivates, such a development in the future.

The paper is structured as follows. In Section 2, we give the basic definitions of Brown algebras and Freudenthal triple systems. In Section 3, we consider group schemes related to reduced Brown algebras. For the invariance group of an FTS arising from such an algebra, we determine its subgroup fixing the identity over arbitrary rings. We moreover determine the automorphism group of reduced Brown algebras when 2 is invertible. After this point, we assume that 6 is invertible. In Section 4, we discuss isotopes of Brown algebras and how they arise from elements of the ``unit sphere'' $\bS_Q$ of the corresponding FTS. Finally, in Section 5, we realize these isotopes as twists by a certain $\E_6$-torsor over $\bS_Q$, with total space a group of type $\E_7$. As an applications, we show that Brown algebras may have non-isomorphic isotopes, and that this even occurs for the split Brown algebra.

\subsection*{Notation} Unless otherwise stated, $R$ denotes an arbitrary unital commutative ring. Rings are assumed associative and commutative, while algebras are in general neither. Both rings and algebras are however always unital, and the unity of an algebra $A$ is denoted by $1$ or $1_A$. An $R$-ring is thus a unital, commutative $R$-algebra. We usually describe a (group) scheme $\bX$ (denoted by a boldface letter) over $R$ by specifying the functor of points; as it is often clear where morphisms are mapped, this amounts to specifying $\bX(S)$ for each $R$-ring $S$.

All unadorned tensor products are over $R$. If $M$ is an $R$-module, we will use the notation $M_S$ for the $S$-module $M\otimes S$, and sometimes denote $m\otimes 1_S$ by $m_S$ for $m\in M$. To avoid confusion, we use superscripts rather than subscripts in most other contexts.

We will often be concerned the $R$-module $R\times R\times A\times A$, where $A$ is an Albert algebra (see Section \ref{Salbert}). We will write this module as
\[\left(\begin{array}{cc}
R&A\\A&R        \end{array}\right)=
\left\{\left(\begin{array}{cc}
r&a\\b&s        \end{array}\right) \mid r,s\in R, a,b\in A\right\}.
\]
We will use the standing notation
\[\begin{array}{llll}x=\left(\begin{array}{cc}
r&a\\b&s        \end{array}\right) &\text{and}& x'=\left(\begin{array}{cc}
r'&a'\\b'&s'        \end{array}\right),\end{array}
                                \]
and we fix the notation
\[\begin{array}{lllll} e=\left(\begin{array}{cc}
1&0\\0&0        \end{array}\right),& f=\left(\begin{array}{cc}
0&0\\0&1        \end{array}\right)&\text{and} & j=\left(\begin{array}{rr}
1&0\\0&-1        \end{array}\right).\end{array}                                                                      \]

\section{Basic Definitions}
\subsection{Albert Algebras}\label{Salbert}
Let $A$ be an Albert algebra over $R$. Recall that this is a cubic Jordan algebra over $R$ whose underlying module is projective of constant rank 27, and for which $A\otimes_Rk$ is simple for every $R$-field $k$. We refer to \cite{Pet} for details on Albert algebras over rings. In particular, any Albert algebra $A$ is endowed with a cubic form $N=N_A:A\to R$ (the norm) and a quadratic map $\sharp:A\to A$ (the adjoint). We denote by $T:A\times A\to R$ the bilinear trace of $N$, given by
\[T(x,y)=N(1,x)N(1,y)-N(1,x,y)\]
where the partial linearization $N(,)$ is quadratic in the first argument and linear in the second, and is given by
\[N(x+ty)=N(x)+tN(x,y)+t^2N(y,x)+t^3N(y)\]
where $t$ is a formal variable; moreover,
\[N(x,y,z)=N(x+y+z)-N(x+y)-N(x+z)-N(y+z)+N(x)+N(y)+N(z)\]
is the full linearization of $N$. It is known that
\[N(x,y,z)=T(x,y\times z),\]
where the cross-product
\[y\times z=(y+z)^\sharp-y^\sharp-z^\sharp\]
is the linearization of the quadratic adjoint. When $S$ is an $R$-ring and $\rho\in\bGL(A)(S)$, we set $\rho^\dagger=(\rho^t)^{-1}$, where the transpose is with respect to the (non-degenerate) bilinear trace $T_S$ of $A_S$. Equivalently,
\[T_S(\rho(x),\rho^\dagger(y))=T_S(x,y)\]
for all $x,y\in A_S$.

\subsection{Brown Algebras}
In this section we let $A$ be an Albert algebra over $R$.
\begin{Def} The \emph{reduced Brown algebra associated to $A$} is the $R$-algebra \[B^A=\left(\begin{array}{cc}
R&A\\A&R        \end{array}\right)\]
with multiplication
\[\left(\begin{array}{cc}
r&a\\b&s        \end{array}\right)\left(\begin{array}{cc}
r'&a'\\b'&s'        \end{array}\right)=
\left(\begin{array}{cc}
rr' + T(a,b')&ra'+s'a+b\times b'\\r'b+sb'+a\times a'& ss'+T(a',b)        \end{array}\right)\]
and involution $^*$ given by
\[\left(\begin{array}{cc}
r&a\\b&s        \end{array}\right)^*=\left(\begin{array}{cc}
s&a\\b&r        \end{array}\right).\]
\end{Def}

If $A=A^s$ is the split Albert algebra over $R$, we will denote $B^A$ by $B^s$ and call it the \emph{split Brown algebra over $R$}. It is obtained by base change from the split Brown algebra $B_0$ over $\Z$, since indeed, $A^s=A_0\otimes_\Z R$, where $A_0$ is the split Albert algebra over $\Z$, whence $B^s=B_0\otimes_\Z R$.

\begin{Def} A \emph{Brown algebra over $R$} is an $R$-algebra with involution $B=(B,^*)$ whose underlying module is projective of rank 56, and that is fppf-locally isomorphic to the split Brown algebra in the sense that $B\otimes S\simeq B^s\otimes S$ for some faithfully flat $R$-ring $S$.
\end{Def}

\begin{Rk} Unless otherwise stated, any (iso)morphism between algebras with involution is assumed to be an (iso)morphism of algebras with involution. 
\end{Rk}

We define the triple product of a Brown algebra $B$ by
\begin{equation}\label{Ebracket}
\{x,y,z\}=(xy^*)z+(zy^*)x-(zx^*)y 
\end{equation}
and use the notation
\begin{equation}\label{EVU}
U_{x,z}y=V_{x,y}z=\{x,y,z\}, 
\end{equation}
with the abbreviation $U_x:=U_{x,x}$.

\begin{Rk}\label{Rcons} Note that the triple product defines a new multiplication $\circ$ on $B$ by 
\[x\circ y=\{x,1,y\}=xy+y(x-x^*).\]
In \cite{AH}, this algebra structure is referred to as the conservative algebra associated to $B$. We will come back to this when discussing isotopes of Brown algebras.
\end{Rk}

\subsection{Freudenthal Triple Systems} 
As explained in the introduction, we will first recall the definition of Freudenthal triple systems from \cite{GPR}. We start with a reduced Brown algebra $B^A$ associated to an Albert algebra $A$ over $R$, and consider the bilinear form $b$ on $B^A$ defined by
\begin{equation}\label{Eb}
b(x,x')=rs'-r's+T(a,b')-T(a',b). 
\end{equation}
This is a non-degenerate bilinear form, since so is $T$ on $A$. Next, a 4-linear form $\Psi_A:B^A\times B^A\times B^A\times B^A\to R$ is defined as follows. Start with
\begin{equation}\label{Eq}
q(x)=4rN(a)+4sN(b)-4T(a^\sharp,b^\sharp)+(rs-T(a,b))^2. 
\end{equation}
If $R=\Z$ and $A=A^s$, let $x_i\in B^A$ and $t_i$ be formal variables for $i=1,2,3,4$. The coefficient of $t_1t_2t_3t_4$ in $q(\sum_i t_ix_i)$ is equal to $2\Theta$ for a 4-linear form $\Theta$, and the 4-linear form $\Psi_A$ is defined by $2\Psi_A$ mapping $(x_1,x_2,x_3,x_4)$ to
\[\Theta(x_1,x_2,x_3,x_4)+b(x_1,x_2)b(x_3,x_4)-b(x_1,x_3)b(x_2,x_4)+b(x_1,x_4)b(x_2,x_3).\]
This is extended to split Brown algebras over arbitrary rings by base change, and to reduced Brown algebras by descent. Thus one obtains, for any Albert algebra $A$ over $R$, a triple $(B^A,\Psi_A,b_A)$. We call this \emph{the Freudenthal triple system associated to $A$}, and denote it by $Q^A$. When $A=A^s$, we call $Q^A$ the \emph{split FTS} and denote it by $Q^s$. Note that $B^A$ is viewed as an $R$-module, with no regard to its structure of a Brown algebra.

\begin{Def} A \emph{Freudenthal triple system} or \emph{FTS} over $R$ is a triple $(B,\Psi,b)$ where $B$ is an $R$-module, $\Psi$ is a 4-linear form on $B$ and $b$ is an alternating bilinear form on $B$, such that $(B\otimes S,\Psi\otimes S,b\otimes S)\simeq(B^A,b_A,\Psi_A)$ for a faithfully flat $R$-ring $S$ and an Albert algebra $A$ over $S$.
\end{Def}

\begin{Rk} Note that, possibly after a further faithfully flat extension $S$, this is equivalent to requiring that $(B\otimes S,\Psi\otimes S,b\otimes S)$ be split.
\end{Rk}

\begin{Rk}\label{Rsix}
If $6\in R^*$, then we remark with \cite{GPR} that from $\Psi$ and $b$ one can recover $\Theta$ and thence $q$ as $q(x)=\tfrac{1}{12}\Theta(x,x,x,x)$. Further, $q$ defines a trilinear form $t$ via its linearization $\Theta$ and the non-degenerate form $b$; namely
\[\Theta(w,x,y,z)=b(w,t(x,y,z)).\]
Thus when $6\in R^*$, then the triple $(B,t,b)$ determines the FTS $(B,\Psi,b)$, and conversely. We shall make frequent use of this.
\end{Rk}

\subsection{Brown algebras and Freudenthal triple systems}\label{SBFTS}
In this subsection, we assume that $6\in R^*$. In \cite{Gar}, FTSs were constructed from any given Brown algebra structure, up to the choice of a skew-symmetric element with invertible square. This construction generalizes, with little modification, to the ring case, as follows.

Let $B$ be a Brown algebra over $R$, and take skew-symmetric elements $z,z'\in B$. Then $z'z=\mu 1_B$ for some $\mu\in R$. (To see this, take a faithfully flat $R$-ring $S$ such that $B_S$ is reduced. Observe that in a reduced Brown algebra, any skew-symmetric element is a multiple of $j$, and thus the product of two such elements is a scalar multiple of the identity. Then $v=z'z$ satisfies
\[v\otimes 1_S=(z'\otimes 1_S)(z\otimes1_S)\in S(1_B\otimes 1_S).\]  Thus the inclusion $R1_B\to R1_B+Rv$ becomes surjective after a faithfully flat extension. Hence $v\in R1_B$ as desired.)

If now $B$ has a skew-symmetric element $z$ with $z^2\in R^*1_B$, then one defines a non-degenerate bilinear form $b_z:B\times B\to R$ by
\[b_z(x,y)=(xy^*-yx^*)z\]
(where we have identified $R$ with $R1_B$), and a trilinear map $t_z:B\times B\times B\to B$ by
\[t_z(x,y,w)=2\{x,zy,w\}-b_z(y,w)x-b_z(y,x)w-b_z(x,w)y.\]

\begin{Ex} If $B=B^A$ and $z=j$, then $b_j$ coincides with the bilinear form $b$ from \eqref{Eb}, and $t_j$ with the trilinear map $t$ from Remark \ref{Rsix}. Thus in this case one recovers $Q^A$.\end{Ex}

Returning to the general case, we need to show that $(B,t_z,b_z)\otimes S \simeq (B^A,t_j,b_j)$ for some faithfully flat $R$-ring $S$ and some Albert algebra $A$ over $S$. We can indeed find such $S$ and $A$ with an isomorphism $\varphi:B_S\to B^A$ of $S$-algebras with involution. Then $\varphi(z\otimes 1_S)=\lambda j$ for some $\lambda\in S$, and $\lambda$ is invertible since $\lambda^2=z^2\otimes 1$. From this one computes that $\varphi$ maps $(B,t_z,b_z)\otimes S$ to $(B^A,\lambda t_j,\lambda b_j)$. Composing $\varphi$ with the map
\[\left(\begin{array}{cc}
r&a\\b&s        \end{array}\right)\mapsto\left(\begin{array}{cc}
\lambda^{-1}r&\lambda a\\b&\lambda^2s        \end{array}\right)\]
(see \cite[Example 16.9]{GPR}) one obtains the desired isomorphism.

We shall call $(B,t_z,b_z)$ the \emph{FTS associated to $B$ and $z$} and denote it by $Q(B,z)$. In case $B=B^A$, one can always choose $z=j$, in which case the above example shows that one retrieves $Q^A$. In general, we do not know which non-reduced Brown algebras admit skew-symmetric elements with invertible square. By \cite{Gar}, this is always the case if $R$ is a field (of characteristic not 2 or 3).

\section{Exceptional Groups}
\subsection{The Invariance Group of an FTS}
Morphisms, isomorphisms and automorphisms of Freudenthal triple systems are defined in the natural way. Given an FTS $Q=(B,\Psi,b)$, the $R$-group scheme of all automorphisms of $Q$ is called the \emph{invariance group of $Q$} and denoted $\bInv(Q)$. The following was proved in \cite{GPR}.

\begin{Prp}\label{PGPR} Let $Q=(B,\Psi,b)$ be an FTS over $R$. Then $\bInv(Q)$ is a simple, simply connected group scheme of type $\E_7$. If $Q=Q^A$ for an Albert algebra $A$ with cubic norm $N$, then the map $\iota: \bIsom(N)\to\bGL(B)$ defined, for each $R$-ring $S$, by
\begin{equation}\label{Eiota}
\iota_S(\rho):=\widehat\rho:\left(\begin{array}{cc}
r&a\\b&s        \end{array}\right)\mapsto\left(\begin{array}{cc}
r&\rho(a)\\\rho^\dagger(b)&s        \end{array}\right)
\end{equation}
defines an injective group homomorphism $\bIsom(N)\to\bInv(Q)$.
\end{Prp}

\subsection{The Subgroup Fixing the Unity: the Reduced Case}
Let $A$ be an Albert algebra over $R$. The $R$-module $B^A$ contains a distinguished element $\diag(1,1)$, which is the identity element of the Brown algebra $B^A$. The aim of this section is to determine the automorphisms of $Q^A$ fixing the identity.

\begin{Def}\label{DG} The subgroup scheme $\bInv^1(Q^A)$ of $\bInv(Q^A)$ is defined by
\[\bInv^1(Q^A)(S)=\{\varphi\in\bInv(Q^A)(S)\mid \varphi(1_{(B^A)_S})=1_{(B^A)_S}\}\]
for each $R$-ring $S$. 
\end{Def}

We will show that this is the group $\bIsom(N_A)$. This requires a series of lemmata.

\medskip

For the first lemma, let $A$ be an Albert algebra over $R$ and denote by $\Delta$ the full linearization of its cubic norm $N$. Thus as above,
\[\Delta(x,y,z)=T(x,y\times z).\]
Observe that $\Delta(x,x,x)=6N(x)$. If $\bGL(A)^\Delta$ denotes the subgroup of $\bGL(A)$ preserving $\Delta$,  this implies that $\bIsom(N)$ is a subgroup of $\bGL(A)^\Delta$, and if $6\in R^*$, then the two groups coincide. Using a result by Seshadri, this can be extended to $R=\Z$ as follows.

\begin{Lma}\label{L1} Let $A$ be an Albert algebra over $\Z$, with cubic norm $N$ and full linearization $\Delta$. Then $\bGL(A)^\Delta=\bIsom(N)$. \end{Lma}

\begin{proof} By the above, what needs to be shown is that $\bGL(A)^\Delta$ fixes $N$. Now, $N$ is an element of the $\Z$-module $V=\mathrm{Pol}(A,\Z)$ of polynomial laws from $A$ to $\Z$, in the sense of \cite[13.2]{PR}. An element of $V$ is a natural transformation from $A\otimes -$ to $\Z\otimes -$, both viewed as functors from ($\Z$-)rings to $\Z$-modules. The group $\bGL(A)$, and thus $\bGL(A)^\Delta$, acts on $V$ by precomposition. Then since $6\in\Q^*$ we have \[N\otimes 1_\Q\in(V\otimes \Q)^{\bGL(A)^\Delta\times \Q}=V^{\bGL(A)^\Delta}\otimes \Q\]
where the equality is due to \cite[Lemma 2]{Ses} since $\Q$ is flat over the noetherian ring $\Z$. By flatness, therefore, $\bGL(A)^\Delta$ fixes $N$. \end{proof}

Next we need another application of Seshadri's Lemma to prove, over $\Z$, a result that can be obtained, using the trilinear form of Remark \ref{Rsix}, over any ring where $6$ is invertible.

\begin{Lma}\label{L2} Let $R=\Z$ and $B=B^A$ for an Albert algebra $A$ over $\Z$. Then $e,f\in B^{\bInv^1(Q)}$. 
\end{Lma}

\begin{proof} First we show that $\bInv^1(Q)\times \Q$ fixes $e_\Q:=e\otimes 1_\Q$ and $f_\Q:=f\otimes 1_\Q$. Indeed, for any $\Q$-ring $S$ and any $\varphi\in\bInv^1(Q)(S)$, 
\[\varphi(e_S+f_S)=\varphi(1_{B_S})=1_{B_S}=e_S+f_S.\]
Moreover, $\varphi$ commutes with $t_S$, so from $t(1_B,1_B,1_B)=3(e-f)$ we get
\[3\varphi(e_S-f_S)=\varphi(t_S(1_{B_S},1_{B_S},1_{B_S}))=t_S(\varphi(1_{B_S}),\varphi(1_{B_S}),\varphi(1_{B_S}))=3(e_S-f_S),\]
and since $6\in S^*$ this implies that $\varphi(e_S)=e_S$ and $\varphi(f_S)=f_S$. Thus \[e_\Q,f_\Q\in(B\otimes\Q)^{\bInv^1(Q)\times\Q}=B^{\bInv^1(Q)}\otimes\Q,\]
where the equality is due to \cite[Lemma 2]{Ses} as in the previous lemma, and by flatness we conclude that $\bInv^1(Q)$ fixes $e$ and $f$. 
\end{proof}

The next lemma provides the final ingredient into characterizing $\bInv^1(Q)$.

\begin{Lma}\label{L3} Let $S$ be an $R$-ring and assume that $\varphi\in\bInv(Q)(S)$ fixes $e$ and $f$. Then $\varphi$ is of the form $\widehat\rho$ for some $\rho\in\bGL(A)(S)$.
\end{Lma}

The map $\widehat\rho$ was defined in \eqref{Eiota}.

\begin{proof} Since $\varphi$ fixes $Re\oplus Rf$, it maps its orthogonal complement with respect to the bilinear form $b$ to itself. This orthogonal complement consists of the off-diagonal part of $B^A$. Thus for all $a,b\in A$
\begin{equation}\label{E0}\begin{array}{lll}
   \varphi\left(\begin{array}{cc}
0&a\\0&0        \end{array}\right)=\left(\begin{array}{cc}
0&a'\\a''&0        \end{array}\right) &\text{and} &   \varphi\left(\begin{array}{cc}
0&0\\b&0        \end{array}\right)=\left(\begin{array}{cc}
0&b''\\b'&0        \end{array}\right) 
  \end{array}\end{equation}
for some $a',a'',b',b''\in A_S$. We will first show that $a''=b''=0$; the proof is a technical adaptation of the proof of \cite[Lemma 12]{Bro}. We define the trilinear form $\widetilde t$ on $B^A$ defined implicitly by
\[\Psi(w,x,y,z)=b(w,\widetilde t(x,y,z)),\]
so that $\varphi\in\bInv(Q)(S)$ commutes with $\widetilde t$. (This form is used in lieu of the form $t$ from Remark \ref{Rsix}, since we are not assuming 2 or 3 to be invertible.) The key step is to note that
\begin{equation}\label{E1}
\widetilde t\left(\left(\begin{array}{cc}
0&a\\b&0        \end{array}\right),e,f\right)=\left(\begin{array}{cc}
0&a\\0&0        \end{array}\right)
\end{equation}
and
\begin{equation}\label{E2}
\widetilde t\left(\left(\begin{array}{cc}
0&a\\b&0        \end{array}\right),f,e\right)=\left(\begin{array}{cc}
0&0\\-b&0        \end{array}\right),
\end{equation}
which follows from explicit computations using the definitions of $\Psi$ and $b$. Now on the one hand, \eqref{E2} gives
\[\widetilde t\left(\varphi\left(\begin{array}{cc}
0&a\\0&0        \end{array}\right),\varphi(f),\varphi(e)\right)=\varphi \widetilde t\left(\left(\begin{array}{cc}
0&a\\0&0        \end{array}\right),f,e\right)=\varphi(0)=0,
\]
and on the other, by \eqref{E0} and \eqref{E2} and the fact that $\varphi$ fixes $e$ and $f$,
\[\widetilde t\left(\varphi\left(\begin{array}{cc}
0&a\\0&0        \end{array}\right),\varphi(f),\varphi(e)\right)=\widetilde t\left(\left(\begin{array}{cc}
0&a'\\a''&0        \end{array}\right),f,e\right)=\left(\begin{array}{cc}
0&0\\-a''&0        \end{array}\right).
\]
Thus $a''=0$. Arguing similarly with \eqref{E1} instead of \eqref{E2} yields $b''=0$, as desired.

Thus \eqref{E0} gives
\[\begin{array}{lll}
   \varphi\left(\begin{array}{cc}
0&a\\0&0        \end{array}\right)=\left(\begin{array}{cc}
0&\rho(a)\\0&0        \end{array}\right) &\text{and} &   \varphi\left(\begin{array}{cc}
0&0\\b&0        \end{array}\right)=\left(\begin{array}{cc}
0&0\\ \widetilde\rho(b)&0        \end{array}\right) 
  \end{array}\]
for some $\rho,\widetilde\rho\in\bGL(A)$. Since $\varphi$ preserves the bilinear form $b$ it follows that $\widetilde\rho=\rho^\dagger$, and the proof is complete.

\end{proof}

\begin{Thm}\label{TG1} Let $A$ be an Albert algebra over $R$ with cubic norm $N$, and let $Q=Q_A$. Then $\bInv^1(Q)\simeq\bIsom(N)$.
\end{Thm}

\begin{proof} First assume that $R=\Z$. The map $\iota$ of Proposition \ref{PGPR} defines a homomorphism of algebraic groups $\bIsom(N)\to\bInv(Q)$, which by construction factors through the inclusion $\bInv^1(Q)\to\bInv(Q)$. Let $S$ be
a ring and $\varphi\in\bInv^1(Q)(S)$. We need to show that $\varphi=\widehat\rho$ for some $\rho\in \bIsom(N)(S)$. By Lemma \ref{L2}, $\varphi$  fixes the diagonal elements of $B^A$. Thus by Lemma \ref{L3},  $\varphi=\widehat\rho$ for some $\rho\in \bGL(A)(S)$. It remains to be shown that $\rho\in\bIsom(N)(S)$.  From the definition of $\Psi$ one computes
\[\Psi\left(\left(\begin{array}{cc}
1&0\\0&0        \end{array}\right),\left(\begin{array}{cc}
0&a\\0&0        \end{array}\right),\left(\begin{array}{cc}
0&b\\0&0        \end{array}\right),\left(\begin{array}{cc}
0&c\\0&0        \end{array}\right)\right)=T(a,b\times c).\]
Applying $\varphi$ and using the fact that $\varphi$ preserves $\Psi$, this equals $T(\rho(a),\rho(b)\times \rho(c))$. Thus by Lemma \ref{L1}, $\rho\in\bIsom(N)(S)$. Similarly one shows that $\rho^\dagger\in\bIsom(N)(S)$, which completes the proof over $\Z$. The statement then holds for the split Albert algebra $A^s$ over an arbitrary ring $R$, since $A^s$ is obtained from the split Albert algebra over $\Z$ by base change. Since any Albert algebra is split by a faithfully flat extension, the claim follows by faithfully flat descent.
\end{proof}

\subsection{The Automorphism Group of a Reduced Brown Algebra}

Fix an Albert algebra $A$ over $R$ and let $B=B^A$ be the associated Brown algebra.

\begin{Rk}\label{Rhat} If $A$ and $A'$ are Albert algebras and $\rho:A\to A'$ a norm isometry, then
\[\widehat\rho:\left(\begin{array}{cc}
r&a\\b&s        \end{array}\right)\mapsto\left(\begin{array}{cc}
r&\rho(a)\\\rho^\dagger(b)&s        \end{array}\right)\]
is a homomorphism of Brown algebras. Indeed, by construction, the involution on $B$ commutes with $\widehat\rho$. Moreover,
\[\widehat\rho(x)\widehat\rho(x')=
\left(\begin{array}{cc}
rr' + T(\rho(a),\rho^\dagger(b'))&r\rho(a')+s'\rho(a)+\rho^\dagger(b)\times \rho^\dagger(b')\\r'\rho^\dagger(b)+s\rho^\dagger(b')+\rho(a)\times \rho(a')& ss'+T(\rho(a'),\rho^\dagger(b))        \end{array}\right)\]
which is equal to 
\[\widehat\rho\left(\begin{array}{cc}
rr' + T(a,b')&ra'+s'a+b\times b'\\r'b+sb'+a\times a'& ss'+T(a',b)        \end{array}\right)=\widehat\rho(xx'):\]
indeed, we have the identities $T(\rho(x),\rho^\dagger(y))=T(x,y)$ and $\rho^\dagger(x\times y)=\rho(x)\times\rho(y)$; the latter follows, by non-degeneracy of $T$, from the invariance of $T(x\times y,z)=N(x,y,z)$ under norm isometries. Thus $\widehat\rho$ is a homomorphism, which is injective (resp.\ surjective) if and only if $\rho$ is.
 \end{Rk}

\begin{Rk}
Note that if 2 is invertible in $R$, and $\rho\in\bAut(B)(S)$ for an $R$-ring $S  $, the element $j=\diag(1,-1)\in B_S$ satisfies $\rho(j)=\pi(\rho)j$ for a unique $\pi(\rho)\in\bmu_2(S)$. It is straightforward to check that this defines a homomorphism of affine group schemes $\pi:\bAut(B)\to\bmu_2$.
\end{Rk}

Next we relate $\bAut(B)$ to the group $\bIsom(N_A)$ via the map $\iota$ of Proposition \ref{PGPR}.

\begin{Prp}\label{Pincl} The map $\iota$
defines a homomorphism of affine group schemes $\bIsom(N_A)\to\bAut(B)$. If $2\in R^*$, then the sequence
\[\xymatrix{\mathbf{1}\ar[r]&\bIsom(N_A)\ar[r]^-{\iota}&\bAut(B)\ar[r]^-{\pi}&\bmu_2\ar[r]&\mathbf{1}}
\]
is split exact.
\end{Prp}

\begin{proof} The map $\iota$ is well-defined by Remark \ref{Rhat} with $A'=A$, and defines an injective group homomorphism $\bIsom(N_A)\to\bAut(B)$, which factors through the kernel of $\pi$ by construction. (This also follows from the connectedness of $\bIsom(N_A)$.) 

\medskip

Next we assume that $2\in R^*$ and show that $\iota$ is an isomorphism onto $\ker(\pi)$. Since $\bIsom(N_A)$ is smooth, we may, by \cite[$_4$ 17.9.5]{EGAIV}, assume that $R=k$ is a field, and by descent of isomorphisms, that it is further algebraically closed. In particular, we may assume that $A$ is split, and thus $A=A_0\otimes_\Z k$ for the split Albert algebra $A_0$ over $\Z$. Let $S$ be a $k$-ring\footnote{We will drop the subscript $S$ in the base change to $S$ to simplify notation.} and $\varphi\in\bAut(B)(S)$ with $\pi(\varphi)=1$, so that  $\varphi(j)=j$. Since also $\varphi(1)=1$, and $\{1,j\}$ spans the diagonal elements of $B$, this implies that $\varphi$ fixes any diagonal element of $B$. If
\[x\in\left(\begin{array}{cc}
0&A\\0&0        \end{array}\right),\]
 then  $ex=x$, and since $\varphi(e)=e$ by the above, we conclude that
 \[\varphi(x)\in\left(\begin{array}{cc}
0&A\\0&0        \end{array}\right).\]
Likewise, using $f$, we find that $\varphi$ stabilizes
 \[\left(\begin{array}{cc}
0&0\\A&0        \end{array}\right).\]
Therefore,
 \[\varphi\left(\begin{array}{cc}
r&a\\b&s        \end{array}\right)=\left(\begin{array}{cc}
r&\rho(a)\\\widetilde\rho(b)&s        \end{array}\right),\]
for some linear bijections $\rho,\widetilde\rho:A\to A$. 

Next we show that $\widetilde\rho=\rho^\dagger$. With $a,b\in A$, we compute
\[\varphi\left(\left(\begin{array}{cc}
0&a\\0&0        \end{array}\right)\left(\begin{array}{cc}
0&0\\b&0        \end{array}\right)\right)=\varphi\left(\begin{array}{cc}
T(a,b)&0\\0&0        \end{array}\right)=\left(\begin{array}{cc}
T(a,b)&0\\0&0        \end{array}\right),\]
which is equal to
\[\varphi\left(\begin{array}{cc}
0&a\\0&0        \end{array}\right)\varphi\left(\begin{array}{cc}
0&0\\b&0        \end{array}\right)=\left(\begin{array}{cc}
0&\rho(a)\\0&0        \end{array}\right)\left(\begin{array}{cc}
0&0\\ \widetilde\rho(b)&0        \end{array}\right)=\left(\begin{array}{cc}
T(\rho(a),\widetilde\rho(b))&0\\0&0        \end{array}\right).\]
Since $a$ and $b$ are arbitrary, this implies that $\widetilde\rho=\rho^\dagger$. 

If we next apply $\varphi$ to both sides of the identity
\[\left(\begin{array}{cc}
0&0\\a&0        \end{array}\right)\left(\left(\begin{array}{cc}
0&0\\b&0        \end{array}\right)\left(\begin{array}{cc}
0&0\\c&0        \end{array}\right)\right)=\left(\begin{array}{cc}
0&0\\0&T(a,b\times c)        \end{array}\right),\]
and use the fact that $\varphi$ is a homomorphism, we get
\[T(\rho(a),\rho(b)\times\rho(c))=T(a,b\times c),\]
i.e.\ $\rho\in\bGL(A)^\Delta(S)$, where $\rho\in\bGL(A)^\Delta$ is the $k$-group scheme defined in the previous section. Since $A=A_0\otimes_\Z k$, this is equal to $\bGL(A_0)^\Delta(S)$, where $\bGL(A_0)^\Delta$ is a $\Z$-group scheme. By Lemma \ref{L1}, we have
\[\bGL(A_0)^\Delta(S)=\bIsom(N_{A_0})(S)=\bIsom(N_A)(S)\]
since $A=A_0\otimes_\Z k$ and $S$ is a $k$-ring. Thus $\rho\in \bIsom(N_A)(S)$. This completes the proof of the exactness at $\bAut(B)$.

\medskip

To finally show that the sequence is split exact, we exhibit a splitting $\Sigma:\bmu_2\to \bAut(B)$, defined, for any $R$-ring $S$ and any $\epsilon\in\bmu_2(S)$, by
\[\Sigma_S(\epsilon):\left(\begin{array}{cc}
r&a\\b&s        \end{array}\right)\mapsto \tfrac{1+\epsilon}{2}\left(\begin{array}{cc}
r&a\\b&s        \end{array}\right)+\tfrac{1-\epsilon}{2}\left(\begin{array}{cc}
s&b\\a&r        \end{array}\right).\]
Straightforward calculations show that $\Sigma_S(\epsilon)\in\bAut(B)(S)$ for each $\epsilon\in\bmu_2(S)$, and  that $\Sigma_S$ is a group homomorphism with $\pi_S\Sigma_S=\Id_{\bmu_2(S)}$. Functoriality is clear, and the proof is complete.
\end{proof}

\begin{Cor}\label{Cisom} Assume that $2\in R^*$ and let $A$ be an Albert algebra over $R$. The group $\bAut(B^A)$ is a semidirect product $\bIsom(N_A)\rtimes\bmu_2$. The identity component of $\bAut(B^A)$ is isomorphic to $\bIsom(N_A)$, and thus a simply connected simple group of type $\mathrm{E}_6$. Two reduced Brown algebras $B^A$ and $B^{A'}$ are isomorphic if and only if the norms of $A$ and $A'$ are isometric.
\end{Cor}

\begin{proof} The two first statements are immediate from the previous proposition and \cite[Proposition 2.1]{Als}. For the third, Remark \ref{Rhat} establishes the ``if''-direction. Moreover, the map $\iota$ induces a map
\[\iota^*: H_{\mathrm{fppf}}^1(\bIsom(N_A))\to H_{\mathrm{fppf}}^1(\bAut(B^A))\]
in cohomology. If $B^{A'}$ is isomorphic to $B^A$, then the isometry class of $N_{A'}$ in $H_{\mathrm{fppf}}^1(\bIsom(N_A))$ is in the kernel of $\iota^*$. By \cite[Proposition 2.4.3]{Gil2}, this kernel is in bijection with the $\bAut(B)(R)$-orbits of $\bX(R)$, where  $\bX=\bAut(B)/\bIsom(N_A)$ is the fppf-quotient. Since $\bAut(B)\simeq\bIsom(N_A)\rtimes\bmu_2$ is isomorphic to $\bIsom(N_A)\times\bmu_2$ as an $\bIsom(N_A)$-torsor over $\bX$, it admits a section. Thus $\bX(R)$ consists of a unique $\bAut(B)(R)$-orbit, and the sought kernel is trivial. This completes the proof. \end{proof}

\begin{Rk} The proof of Proposition \ref{Pincl} shows over any ring $R$ (not assuming $2\in R^*)$, $\iota$ defines an isomorphism between $\bIsom(N_A)$ and the group of all automorphisms of $B^A$ fixing $e$ and $f$. Denoting this group by $\bAut(B)^{(e,f)}$ we get, together with Theorem \ref{TG1}, the isomorphisms
\[\bInv^1(Q^A)\simeq\bAut(B^A)^{(e,f)}\simeq\bIsom(N_A).\]
\end{Rk}

\section{Isotopes}

\begin{Def} An \emph{isotopy} between two Brown algebras $B$ and $B'$, with triple products denoted by $\{,,\}$ and $\{,,\}'$, respectively, is a pair of invertible $R$-linear maps $\varphi,\psi: B\to B'$ such that
\begin{equation}\label{Eisot}
\varphi(\{x,y,z\})=\{\varphi(x),\psi(y),\varphi(z)\}'. 
\end{equation}
\end{Def}

\begin{Def} An element $u\in B$ is \emph{conjugate invertible} if there exists $v\in B$ such that $V_{u,v}=\Id$. We then say that $v$ is a \emph{conjugate inverse} of $u$.
\end{Def}

\emph{Henceforth we assume $6\in R^*$.}

\medskip

\begin{Rk} If $V_{u,v}=\Id$ in a Brown algebra $B$, then $V_{v,u}=\Id$. To see this, assume first that $B=B^A$ is reduced. If $V_{u,v}=\Id$, then for any $x\in B$,
\begin{equation}\label{EV}
(uv^*)x+(xv^*)u-(xu^*)v=x. 
\end{equation}
Applying this to $x=1$ yields
\[uv^*=1+u^*v-v^*u\]
Now, $u^*v-v^*u$ is skew-symmetric and thus equal to $\lambda j$ for some $\lambda\in R$. Thus $uv^*=1+\lambda j$, and applying the involution to both sides we get $vu^*=1-\lambda j$. We want to show that $V_{v,u}x=x$, which is equivalent to
\[(vu^*)x-x=(xv^*)u-(xu^*)v.\] 
\end{Rk}
The left hand side is equal to $-\lambda j$ since $vu^*=1-\lambda j$, while the right hand side is equal to $x-(uv^*)x$ by \eqref{EV}. But this is also equal to $-\lambda j$. This proves the statement in the reduced case, and the general case follows by descent, in view of the identity $V_{x,y}\otimes\Id_S=V_{x\otimes 1_S,y\otimes 1_S}$ for any $R$-ring $S$.

The meaning of this remark is that if $u$ is conjugate invertible with conjugate inverse $v$, then $v$ is conjugate invertible with conjugate inverse $u$.

\begin{Rk}\label{Riso} Following \cite{AH}, we observe that, even in the ring setting, the involution of a Brown algebra can be expressed via the triple product as
\[x^*=2x-\{x,1,1\},\]
and the multiplication as
\[xy=\left\{\begin{array}{ll}
           \{x,1,y\} & \text{\ if\ } x^*=x\\ 
           \{x,1,\tfrac13(y+2y^*)\}-2\{\tfrac13(y+2y^*),1,x\} & \text{\ if\ } x^*=-x.
           \end{array}\right.
\]
Thus to check that a linear map $\rho:B\to B'$ between Brown algebras is a homomorphism, it suffices to check that $\rho(1_B)=1_{B'}$ and
\[\rho U_{x,y}(1)=U'_{\rho(x),\rho(y)}\rho(1),\] since then by the above it is compatible with the involution and the multiplication. (Here $U'$ is the $U$-operator of $B'$.) 
\end{Rk}

\begin{Def}
Let $B$ be a Brown algebra over $R$ and $u\in B$ a conjugate invertible element. The \emph{(principal) isotope of $B$ defined by $u$} is the algebra  $B^{(u)}$ with underlying module $B$, and with
circle product
\[x\circ_uy=\{x,u,y\}.\]
\end{Def}

This is a unital algebra with involution and multiplication obtained from the circle product by Remark \ref{Riso}.

\begin{Rk} The link between isotopy and principal isotopes over fields of characteristic not 2 or 3 is discussed in \cite{AH}. We content ourselves with one motivational remark that easily generalizes to the ring setting as follows.

With $x=y=1_B$, equation \eqref{Eisot} gives $\varphi V_{1_B,1_B}=V_{\varphi(1_B),\psi(1_B)}\varphi$. Since $V_{1_B,1_B}=\Id_B$ and $\varphi$ is invertible, this implies that $\psi(1_B)$ is conjugate invertible with conjugate inverse $\varphi(1_B)$. Thus if $(\varphi,\psi)$ is an isotopy from $B$ to $B'$, then
\[\varphi(x\circ y)=\varphi\{x,1_B,y\}=\{\varphi(x),\psi(1_B),\varphi(y)\}'=\varphi(x)\circ_u' \varphi(y),\]
where $u=\psi(1_B)$, and $'$ indicates working in $B'$. Thus $\varphi$ is an isomorphism from $B$ to the principal isotope of $B'$ defined by $u$. As a consequence, all Brown algebras isotopic to a given Brown algebra are isomorphic to principal isotopes of it. 
\end{Rk}

\subsection{Isotopy and Invariance}
In this section we continue to assume that $6\in R^*$. Our goal is to study isotopes of Brown algebras using the invariance group of Freudenthal triple systems. Recall that to a Brown algebra $B$ over $R$ we have associated one FTS for each skew-symmetric $z\in B$ with invertible square. If $B$ is reduced, the existence of such elements is guaranteed. In the reduced case, an obvious choice is the element $j=\diag(1,-1)$. We start by extending Theorem \ref{TG1} to not necessarily reduced algebras.

\begin{Prp}\label{Pstrong} Let $B$ be a Brown algebra over $R$ containing a skew-symmetric element $z$ with invertible square, and let $Q=Q(B,z)$. Then $\bInv^1(Q)$ is a simply connected simple group scheme of type $\E_6$.
\end{Prp}

The group scheme $\bInv^1(Q)$ is defined in analogy with Definition \ref{DG}; namely

\[\bInv^1(Q)(S)=\{\varphi\in\bInv(Q)(S)\mid \varphi(1_{B_S})=1_{B_S}\}\]
for each $R$-ring $S$. 

\begin{proof} Let $A^s$ be the split Albert algebra over $R$. The FTS associated to $A^s$ is $Q^s=(B^s,t_j,b_j)$. We will prove the claim by showing that $\bInv^1(Q)$ is fppf-locally isomorphic to $\bIsom(N^{A^s})$. There is a a faithfully flat $R$-ring $S$ and an isomorphism
\[\varphi: B\otimes S\to B^s\otimes S\]
of Brown algebras over $S$. Then $\varphi$ defines an isomorphism from $Q_S=(B,t_z,b_z)\otimes S$ to $(B^s,\lambda t_j,\lambda b_j)\otimes S$, with $\lambda\in S^*$, as in Section \ref{SBFTS}. It is then straightforward to check that conjugation by $\varphi$ defines an isomorphism $\bInv^1(Q\otimes S)\to \bInv^1(Q^s\otimes S)$. Since $\bInv^1(Q\otimes S)$ is obtained from $\bInv^1(Q)$ by base change, this shows that $\bInv^1(Q)$ is fppf-locally isomorphic to $\bInv^1(Q^s)$. We conclude with Theorem \ref{TG1}.
\end{proof}

Let $B$ be a Brown algebra over $R$ containing a skew-symmetric element $z$ with invertible square, and let $Q=Q(B,z)$.  We set, for any $R$-ring $S$,
\[\bS_Q(S)=\{x\in B\otimes S\mid q_S(x)=1\},\]
where $q$ is the quartic form $q:x\mapsto \tfrac1{12}b_z(x,t_z(x,x,x))$. This defines an $R$-scheme $\bS_Q$, on which $\bInv(Q)$ acts naturally.

\begin{Prp}\label{Pquot} The fppf-quotient $\bInv(Q)/\bInv^1(Q)$ is representable by a smooth scheme. The map $\Pi:\bInv(Q)\to\bS_Q$ defined by $\rho\mapsto\rho(1_B)$ for any $R$-ring $S$ and $\rho\in\bInv(Q)(S)$, induces an isomorphism $\bInv(Q)/\bInv^1(Q)\to\bS_Q$.
\end{Prp}

We will prove this by reducing to the setup of \cite{Bro}. Note that the form $q$ in \cite{Bro} is a factor $-2$ times the form which we call $q$.

\begin{proof} The action of $\bInv(Q)$ on $\bS_Q$ is transitive on geometric fibres due to \cite[Theorem 3]{Bro}, since indeed, over an algebraically closed field, any Albert algebra is split, hence a fortiori reduced. The stabilizer of $1_B$ is $\bInv^1(Q)$. By Proposition \ref{Pstrong}, the group scheme $\bInv^1(Q)$ is simply connected simple of type $\E_6$. Since thus $\bInv(Q)$ and $\bInv^1(Q)$ are both smooth, so is their quotient by \cite[XVI.2.2 and VIB.9.2]{SGA3}. The induced map is then an isomorphism by \cite[III.3.2.1]{DG}.
\end{proof}

\begin{Rk}\label{Rsurj} As a consequence, the action of $\bInv(Q)$ on $\bS_Q$ is fppf-locally transitive: if $u\in \bS_Q(R)$, then for some $R$-ring $S$, the element $u\otimes 1\in\bS_Q(S)$ equals $\rho(1_{B_S})$ for some $\rho\in\bInv(Q)(S)$.
 
\end{Rk}

 \begin{Lma}\label{Lconj} If $B=B^A$ and $Q=Q^A$ for an Albert algebra $A$, and $u\in \bS_Q(R)$, satisfies $u=\varphi(1)$ for some $\varphi\in\bInv(Q)(R)$, then $U_u$ is invertible and $u$ is conjugate invertible.
 \end{Lma}
 
 \begin{proof} Following \cite{Gar}, we have
 \[2U_{x,z}(jy)=\{x,jy,z\}=t(x,y,z)+b(y,z)x+b(y,x)z+b(x,z)y.\]
 Thus if $\varphi\in\bInv(Q)(R)$, then
 \[\varphi(U_{x,z}(jy))=U_{\varphi(x),\varphi(z)}(j\varphi(y)).\]
 Using $j(jy)=y$ and setting $y'=jy$ we get
 \[\varphi(U_{x,z}(y'))=U_{\varphi(x),\varphi(z)}(\psi(y')),\]
 where $\psi:w\mapsto j\varphi(jw)$. Thus $\varphi U_{x,z}=U_{\varphi(x),\varphi(z)}\psi$, and in particular
 \[U_u=U_{\varphi(1),\varphi(1)}=\varphi U_1\psi^{-1}.\]
 Noting that $U_1(x)=x$ if $x$ is symmetric with respect to the involution, and $U_1(x)=-3x$ if $x$ is skew-symmetric, this shows that $U_u$ is invertible. We further claim that $u$ is conjugate invertible. Indeed,
 \[V_{u,\psi(1)}(x)=U_{\varphi(1),\varphi(\varphi^{-1}(x)}(\psi(1))=\varphi(U_{1,\varphi^{-1}(x)}(1))=\varphi(\varphi^{-1}(x))=x,\]
 whence $V_{u,v}=\Id$ when $v=\psi(1)$. This completes the proof.
 \end{proof}

 Next we use descent to prove the general case.
 
 \begin{Lma}\label{Lconjinv} Let $B$ be a Brown algebra with a skew-symmetric element $z$ with invertible square, and set $Q=Q(B,z)$. Then any $u\in\bS_Q(R)$ is conjugate invertible.
 \end{Lma}

 \begin{proof} First we show that the linear map $U_u:B\to B$ is bijective. This is equivalent to $U_u\otimes \Id_S: B_S\to B_S$ being bijective for some faithfully flat $R$-ring $S$. We can choose $S$ such that $B_S=B^s$ is the split Brown algebra, and $u\otimes1_S=\varphi(1_{B_S})$ for some $\varphi\in\bInv(Q)(S)$, by Remark \ref{Rsurj}. Then by Lemma \ref{Lconj}, $U_{u\otimes 1}$ is bijective and $u\otimes 1$ is conjugate invertible. Since the involution on $B_S$ is $S$-linear, we have $U_u\otimes \Id_S=U_{u\otimes 1}$, which thus shows that $U_u\otimes \Id_S$, and hence by descent $U_u$, is bijective.
 
 To finish the proof we will show that $v=U_u^{-1}(u)$ is a conjugate inverse to $u$. Indeed, let $x\in B$ be arbitrary and set 
 \[y_x=V_{u,v}(x)-x.\]
 We need to show that $y_x=0$, which is equivalent to $y_x\otimes 1_S=0$ for $S$ as before. Now
 \[y_x\otimes 1_S=V_{u,v}(x)\otimes 1_S-x\otimes 1_S=(V_{u,v}\otimes \Id_S)(x\otimes 1_S)-x\otimes 1_S\]
 and since the involution on $B_S$ is $S$-linear, this is equal to 
 \[V_{u\otimes 1_S,v\otimes 1_S}(x\otimes 1_S)-x\otimes 1_S.\]
 But $u\otimes 1_S$ is conjugate invertible with conjugate inverse $\widehat{u\otimes 1_S}$ satisfying
 \[U_{u\otimes 1_S}(\widehat{u\otimes 1_S})=V_{u\otimes 1_S,\widehat{u\otimes 1_S}}(u\otimes 1_S)=u\otimes 1_S,\]
 so $\widehat{u\otimes 1_S}=U_{u\otimes 1_S}^{-1}(u\otimes 1_S)$. Thus we are done if can show that $U_{u\otimes 1_S}^{-1}(u\otimes 1_S)=v\otimes 1_S$, and indeed
 \[U_{u\otimes 1_S}(v\otimes 1_S)=U_u(v)\otimes 1_S=U_uU_u^{-1}(u)\otimes 1_S=u\otimes 1_S,\]
 which completes the proof.
 \end{proof}

 Thus any element of $\bS_Q(R)$ is conjugate invertible, and $\widehat u:=U_u^{-1}(u)$ is a well-defined conjugate inverse of $u$. The next step is the following analogue of a result of \cite{Gar}.
 
 \begin{Lma}\label{Lisot} If $\varphi\in\bInv(Q)(R)$ satisfies $\varphi(1)=u$, then $\varphi$ is an isomorphism $B\to B^{(\widehat u)}$.
 \end{Lma}
 
 \begin{proof} Since $u$ is the unity in  $B':=B^{(\widehat u)}$, it suffices, by Remark \ref{Riso}, to show  the identity
 \[\varphi U_{x,z}(1)=U'_{\varphi(x),\varphi(z)}\varphi(1).\]
 It suffices to check this after base change to a faithfully flat $R$-ring $S$, and we may choose $S$ so that $B=B^A$. Then by Lemma \ref{Lconj} and its proof,
 \[\varphi U_{x,z}(1)=U'_{\varphi(x),\varphi(z)}\psi(1)\]
for some invertible linear map $\psi$. To conclude that $\psi(1)=\varphi(1)$, set $x=z=1$. Then the left hand side is $\varphi(1)=u$. The right hand side is $U'_{u}\psi(1)$, where the operator $U'_{u}$ is a bijection that satisfies $U'_{u}u=u$, since $u$ is the unity of $B'$. This forces $\psi(1)=u$, as desired.
 \end{proof}
 
 The next result extends Corollary \ref{Cisom} beyond reduced Brown algebras.

\begin{Prp} Let $B$ be a Brown algebra with a skew-symmetric element $z$ with invertible square, and set $Q=Q(B,z)$. There is an isomorphism of group schemes $\bInv^1(Q)\to\bAut(B)^\circ$.
\end{Prp}

\begin{proof} From Lemma \ref{Lisot} it follows that the inclusion $\bInv(Q)\to \bGL(B)$ induces a homomorphism $\bInv^1(Q)\to\bAut(B)$, which factors through $\bAut(B)^\circ$ since $\bInv^1(Q)$ is connected. To show that it is an isomorphism we may, since $\bInv^1(Q)$  is smooth, use \cite[$_4$ 17.9.5]{EGAIV} and descent to reduce to the case where $R=k$ is an algebraically closed field. In that case, $B$ is split, hence reduced, and we conclude with Theorem \ref{TG1} and Corollary \ref{Cisom}.
 
\end{proof}

\section{Torsors and twists}
In this section we assume that $6\in R^*$. We consider a Brown algebra $B$ containing a skew-symmetric element $z$ with invertible square, and the corresponding FTS $Q=Q(B,z)$. Proposition \ref{Pquot} has the following immediate consequence.

\begin{Cor} The map $\Pi:\bInv(Q)\to \bS_Q$ defines an $\bInv^1(Q)$-torsor $\bE$ over $\bS_Q$ with respect to the fppf-topology.
\end{Cor}

For each $u\in\bS_Q(R)$, we denote the fibre in $\bE$ by $\bE^u$. This is an $\bAut(B)^\circ$-torsor over $\Spec R$ with respect to the fppf-topology. Our next result concerns the twist of $B$ by $\bE^u$, denoted $\bE^u\wedge B$. We recall the definition step-by-step. For each $R$-ring $S$, consider the set $(\bE^u(S)\times B_S)/\sim$, where the equivalence relation $\sim$ on $\bE^u(S)\times B_S$ is defined by $(\varphi,x)\sim (\psi,y)$ if and only if
\[\exists \gamma\in \bAut(B)^\circ(S) : (\varphi\gamma,x)= (\psi,\gamma(y)).\]
We will denote the equivalence class of $(\varphi,x)\in\bE^u(S)\times B_S$ by $[\varphi,x]$. This quotient set is an $S$-algebra, under the addition
\[ [\varphi,x] + [\psi,y]=[\varphi,x+ \varphi^{-1}\psi(y)],\]
multiplication
\[ [\varphi,x]\cdot[\psi,y]=[\varphi,x\cdot \varphi^{-1}\psi(y)],\]
and $S$-action
\[ \lambda[\varphi,x] =[\varphi,\lambda x]\]
for all $\varphi,\psi\in \pi^{-1}(1_S\otimes u)$, $x,y\in B_S$ and $\lambda\in S$. (Note that the addition and multiplication are understood by observing that the definition of $\sim$ implies that $[\psi,y]=[\varphi,\varphi^{-1}\psi(y)]$, and then defining the operations on the second component after having fixed the first.)

The assignment of the $S$-algebra $(\bE^u(S)\times B_S)/\sim$ to each $R$-ring $S$ defines a presheaf of algebras over $R$. We then let $\bE^u\wedge B$ be the associated sheaf. We now show that these twists are isotopes of $B$.

\begin{Thm}\label{Tiso}
 Let $u\in\bS_Q(R)$. For each $R$-ring $S$ such that $\bE^u(S)\neq\emptyset$, the map 
 \[\begin{array}{cc}\Theta_S:\bE^u(S)\wedge B_S\to B_S^{(\hat u)}, & [\varphi,x]\mapsto \varphi(x)\end{array}\]
 is an isomorphism of $S$-algebras. Thus $(\bE^u\wedge B)(R)$ is canonically isomorphic to $B^{(\hat u)}$.
 \end{Thm}
 
 Here $\hat u=U_u^{-1}(u)$ is conjugate inverse to $u$, as in the proof of Lemma \ref{Lconjinv}.

 \begin{proof} If $\varphi\in\bE^u(S)$, then $\varphi\in\bInv(Q)(S)$ with $\varphi(1)=u$. Thus by Lemma \ref{Lisot}, $\varphi$ is an isomorphism from $B$ to $B^{(\widehat u)}$. Hence
 \[\Theta_S([\varphi,x][\varphi,y])=\Theta_S([\varphi,xy])=\varphi(xy)=\varphi(x)\bullet \varphi(y)=\Theta_S([\varphi,x])\bullet \Theta_S([\varphi,y]), \]
 where $\bullet$ denotes the multiplication in $B^{(\widehat u)}$. Since also
 \[\Theta_S([\varphi,x]+[\varphi,y])=\Theta_S([\varphi,x+y])=\varphi(x+y)=\varphi(x)+ \varphi(y)\]
 and
 \[\Theta_S(\lambda[\varphi,x])=\Theta_S([\varphi,\lambda x])=\varphi(\lambda x)=\lambda\varphi(x), \]
 the map $\Theta_S$ is an isomorphism. The claim follows by sheafification.
 \end{proof}

\subsection{Non-isomorphic isotopes}

The results of the previous sections help answer the question of whether Brown algebras can have non-isomorphic isotopes in the affirmative, even in the case of the split Brown algebra. This follows from the following.
 
 \begin{Prp} Let $A$ be the (split) complex Albert algebra and set $B=B^A$ and  $Q=Q^A$. Then the $\bAut(B)^\circ$-torsor $\bE\to\bS_Q$ is non-trivial.
  
 \end{Prp}
 
 \begin{proof} By  \cite[Lemma 1.4]{Als}, inspired by \cite{Gil1}, it is enough to show that the homotopy group $\pi_n(\bAut(B)^\circ(\mathbb C))$, for some $n$, is not a direct summand of $\pi_n(\bInv(Q)(\mathbb C))$. By Cartan decomposition, $\bAut(B)^\circ(\mathbb C)$ is homeomorphic to $\mathbb R^l\times H$, where $H$ is the compact real Lie group of type $E_6$, and $\bInv(Q)^\circ(\mathbb C)$ is homeomorphic to $\mathbb R^m\times G$, where $G$ is the compact real Lie group of type $E_7$, and $l,m\in\mathbb N$. From \cite{BS} we know that $\pi_9(G)$ is trivial, while $\pi_9(H)\simeq\mathbb Z$. This completes the proof.
\end{proof}

A consequence of this is the following.

\begin{Cor} There exists a smooth $\C$-ring $R$ such that the split Brown algebra over $R$ admits a non-isomorphic isotope.
\end{Cor}

\begin{proof} Let $R$ be the coordinate ring of $\bS_Q$, where $Q=Q^s$ is the (split) complex FTS. By Proposition \ref{Pquot}, this is a smooth $\C$-ring. Let $N$ be the cubic norm of the split Albert algebra $A^s$ over $R$, and let $K$ be the kernel of the map
\[H_{\mathrm{fppf}}^1(\bIsom(N))\to H_{\mathrm{fppf}}^1(\bInv(Q))\]
induced by the inclusion $\bIsom(N)\to\bInv(Q)$. By \cite[Proposition 2.4.3]{Gil2}, $K$ is in bijection with the $\bInv(Q)(R)$-orbits of $\bS_Q(R)$. By the non-triviality of the torsor from the previous proposition, $K$ contains a non-trivial element $\eta$ corresponding to an element $u\in\bS_Q(R)$ that is not in the orbit of $1_{B^s}$. Let $B'$ be the isotope of $B^s$ defined by $\widehat u$.  If $B'\simeq B^s$, then $\eta$ is in the kernel of the map
 \[H_{\mathrm{fppf}}^1(\bIsom(N))\to H_{\mathrm{fppf}}^1(\bAut(B))\]
induced by the map $\iota$ from Proposition \ref{Pincl}. By the proof of Corollary \ref{Cisom}, this kernel is trivial. Thus $B'\not\simeq B^s$.
\end{proof}

\end{document}